\documentclass{amsart}
\usepackage{fancyhdr}

\usepackage{color}

\usepackage[english]{babel}
\usepackage{eucal}
\usepackage{verbatim}
\usepackage{amsmath}
\usepackage{amssymb}
\usepackage{amsfonts}
\usepackage{mathrsfs}
\usepackage{amsthm}
\usepackage{lscape}
\usepackage{hyperref}
\usepackage[all]{xy}
\usepackage{nth}
\usepackage{caption}
\usepackage{mathabx}
\DeclareMathAlphabet{\mathpzc}{OT1}{pzc}{m}{it}
\newcommand{\C}{\boldsymbol{\mathrm{C}}}
\newcommand{\R}{\boldsymbol{\mathrm{R}}}
\newcommand{\F}{\boldsymbol{\mathrm{F}}}
\newcommand{\Z}{\boldsymbol{\mathrm{Z}}}

\newcommand{\St}{\mathbb{S}}

\newcommand{\Spec}{{\rm Spec }}
\newcommand{\Specg}{{\rm Spec }^g}

\newcommand{\e}{\varepsilon}

\newcommand{\Homg}{{\rm Hom}_{gr-alg}}
\newcommand{\Homa}{{\rm Hom}_{alg}}

\newcommand{\Gr}{\mathcal{G}\mathcal{R}}
\newcommand{\nil}{{\rm nil}}
\newcommand{\ch}{{\rm char}}

\newcommand{\kf}{\boldsymbol{\rm k}}

\renewcommand{\_}{\rule{1.2ex}{.5pt}\,}

\newtheorem{theorem}{Theorem}[section]
\newtheorem{lemma}[theorem]{Lemma}
\newtheorem{proposition}[theorem]{Proposition}
\newtheorem{corollary}[theorem]{Corollary}

\theoremstyle{definition}
\newtheorem{definition}[theorem]{Definition}
\newtheorem{example}[theorem]{Example}

\newtheorem{rk}[theorem]{Remark}

\title{Graded Group schemes and graded group varieties}
\author{Camil I. Aponte Rom\'{a}n} 
\email{camili@gmail.com}
\urladdr{www.math.washington.edu/\textasciitilde camili} 
\begin{document}
\begin{abstract}
We define graded group schemes and graded group varieties and develop their theory. Graded group schemes are the graded analogue of affine group schemes and are in correspondence with graded Hopf algebra. Graded group varieties take the place of infinitesimal group schemes. We generalize the result that connected graded bialgebras are graded Hopf algebra to our setting and we describe the algebra structure of graded group varieties. We relate these new objects to the classical ones providing a new and broader framework for the study of graded Hopf algebras and affine group schemes.
\end{abstract}
\maketitle

\tableofcontents

\section{Introduction}

Group schemes and in particular infinitesimal group schemes are very well studied algebraic and geometric objects. Group schemes are a generalization of algebraic groups, and infinitesimal group schemes are an important class of group schemes as they take the role of Lie groups. The main purpose of this paper is to construct the theory of group schemes in the graded realm, that is, when the coordinate rings are graded Hopf algebras and the category is that of graded commutative algebras.  This innocent idea of adding the grading can pose some difficulties, especially when we wish to keep track of the grading in each computation and in each relation we establish. There are also some subtleties arising from the fact not all graded Hopf algebras come from ungraded Hopf algebras. The concepts of graded group schemes and graded group varieties will take the place of group schemes and infinitesimal group schemes, respectively. Defining and studying these two concepts will be the main focus of this paper.

In section \ref{def} we define graded group schemes and set up the notation. We also provide some examples.  We then move to section \ref{grvariety} where we introduce graded group varieties. We give a generalization of the result that connected graded bialgebras are graded Hopf algebra; in the context of graded group varieties. Our result is the algebraic analogue of a well known geometric result regarding projective group schemes. 

Connected graded Hopf algebras are reasonably behaved and understood. It turns out that the coordinate rings of graded group schemes are not necessarily connected, that is, their degree zero part may not be the ground field. In order to circumvent this issue, we construct the algebraic connectivization of a graded group scheme as a connected (in the algebraic sense) graded group scheme, that is, one whose coordinate ring is connected. We establish a relation between a graded group variety and its algebraic connectivization. This allows us to proof one of our main results, which is the classification of graded group varieties.   

Finally in section \ref{secgrcon} we define the graded connected and graded separable components of a graded group schemes. We show that connectivity of graded group schemes relate to the connectivity of the graded spectrum, which is also defined. In section \ref{secclas} we use the constructions from section \ref{secgrcon} in order to give a decomposition of finite graded group schemes in terms of connected and \'{e}tale components. We also give examples of such decompositions.

\section{Definitions and examples}\label{def}

\begin{definition}
Let $\Gr$ be the category of (finitely generated)  graded commutative $\kf$-al\-ge\-bras, where $\kf$ is a field. A
representable functor $G: \Gr \to (groups)$ is called an \textit{affine graded group
scheme}.  We will call them \textit{gr-group schemes} for short. The graded algebra representing $G$ is denoted by $\kf[G]$ and is called the \textit{coordinate algebra} of $G$. We will drop the word affine from now on, as all our gr-schemes will be assumed to be affine.
\end{definition}

As in the ungraded case, by Yoneda's Lemma there is an equivalence of ca\-te\-go\-ries between gr-group schemes and
graded commutative Hopf algebras.

\begin{rk}
A graded algebra is graded commutative if, for $a, b$ homogeneous elements in $A$, we have that $ab = (-1)^{|a||b|}ba$. Note that this way the multiplication map from $m_A: A \otimes A \to A$ is a graded algebra map where the multiplication of $A \otimes A$ is given by $\xymatrixcolsep{4pc}\xymatrix{(A \otimes A) \otimes (A \otimes A) \ar[r]^{A \otimes \tau \otimes A} & (A \otimes A) \otimes (A \otimes A) \ar[r]^-{m_A \otimes m_A} & A \otimes A }$, where $\tau: A \otimes A \to A \otimes A$ is given by $\tau(a \otimes b)  =(-1)^{|a||b|} b \otimes a$. 
\end{rk}

\begin{definition}
We denote $\kf[x_1, \ldots, x_n]^{gr}$ to be the \emph{graded polynomial ring} over $\kf$ in $n$-variable, where $x_ix_j = (-1)^{|x_i||x_j|}x_jx_i$. Note that if $\ch(\kf) \neq 2$, then $x_i^2 = 0$ if $|x_i|$ is odd. 
\end{definition}

\begin{rk} When $\ch(\kf) = 2$, $\kf[x_1, \ldots, x_n]^{gr}$ is just the (ungraded) polynomial ring where the $x_i$'s are graded. For $\ch(\kf) \neq 2$ a standard notation for $\kf[x_1, \ldots, x_n]^{gr}$ is $\kf[y_1, \ldots, y_m] \otimes \Lambda(z_1, \ldots, z_k)$, where the $y_i$'s are evenly graded and the $z_i$'s are oddly graded. The $y_i$'s are in the polynomial part (in the traditional sense) and the $z_i$'s are in the exterior part. We choose not to use this standard notation for the following reason: the exterior part of the graded polynomial rings that arise in the ungraded setting is usually ignored. For example, when studying the cohomology of (ungraded) Hopf algebras or group schemes, people are usually interested in working with a strictly commutative ring. This not what we want to do in our setting. Our objects are graded to begin with, and the odd degree part is as important as the even degree part. With the notation as defined above, we want to convey to the reader that each element in $\kf[x_1, \ldots, x_n]^{gr}$ plays an important role.  We think of $\kf[x_1, \ldots, x_n]^{gr}$ as the graded version of the polynomial ring and the fact that $x_i^2$ may be zero is just a consequence of the graded commutativity.  
\end{rk}

\begin{definition}
 We say that a gr-group scheme $G$ is a \textit{finite gr-group scheme}
if $\kf[G]$ is finite dimensional. In that case we can define $\kf G$ as the graded dual of $\kf[G]$; $\kf G$ is a called the \textit{group algebra} for $G$. 
\end{definition}

\begin{definition}
 We say that a gr-group scheme $G$ is a \textit{positive gr-group scheme}
if $\kf[G]$ is positively graded. That is $\kf[G] = \bigoplus_{i \geq 0} (\kf[G])_i$. 
\end{definition}

\begin{definition}
We say that a gr-group scheme is \emph{algebraically connected} if the zero degree part is $\kf$, that is, $(\kf[G])_0 = \kf$. This is the same as saying that $\kf[G]$ is a connected graded Hopf algebra. 
\end{definition}

\subsection{Examples}

\begin{example}
Consider $\kf[t]^{gr}$ where $|t| = i$ and $\Delta(t) = t \otimes 1 + 1 \otimes t$. Then $\Homg(\kf[t]^{gr}, R ) = (R_i, +)$ 
\end{example}

\begin{example}  Let $R$ be a commutative graded ring. 
Let $\nil^m_i(R) = \{x \in R_i \,\, | \, \, x^m = 0\}$. Let $\ch(\kf) = p > 0$ and $\kf[t]^{gr}/(t^{p^r})$ with $|t| = i$ even, if $p > 2$, and any degree, if $p = 2$. Then $\kf[t]^{gr}/(t^{p^r})$ is a graded commutative Hopf algebra with $\Delta(t) =
t\otimes 1 + 1 \otimes t$ and $\Homg(\kf[t]^{gr}/(t^{p^r}), R) = \nil^{p^r}_i(R)$ is an additive group under
sums. We can see that this group structure
arises from the comultiplication of the algebra, since if $x, y \in
\nil^{p^r}_i(R)$, where $x(t) =  x$ and $y(t) = y$,  then the convolution
product gives that 
 $(x \ast y)(t) = m(x \otimes y)(\Delta(t)) = m(x \ast y)(t \otimes 1+1 \otimes
t) = x(t)y(1)+x(1)y(t)=  x+ y$, where we see $x$ and $y$ as functions in
$\Homg(\kf[t]^{gr}/(t^{p^r}), R)$.
\end{example}

\begin{example}[Dual Steenrod subalgebra $A(1)$]\label{A1}

Consider $A(1) = \F_2[\xi_1, \xi_2]^{gr}/(\xi_1^4, \xi_2^2)$, where $|\xi_1| = 1$ and $|\xi_2| =
3$. The Hopf algebra structure on $A(1)$
is given by $\Delta(\xi_1) = \xi_1 \otimes 1+ 1 \otimes \xi_1$, and
$\Delta(\xi_2) = \xi_2 \otimes 1 + \xi_1^2 \otimes \xi_1 + 1 \otimes \xi_2$.

As sets $\Homg(A(1), R) =  G_1(R) \times G_3(R)$ where $G_1(R) = \nil^{4}_1(R)$ and $G_2(R) = \nil^2_3(R)$ with product given by $(x,u) \ast (y,v) = (x+y, u+v+x^2y)$ and inverse given by $(x, u)^{-1} = (x, u+x^3)$. We will denote the gr-group scheme represented by $A(1)$ by 
$\St_1$.
\end{example}

\section{Graded group variety}\label{grvariety}

The structure of infinitesimal group schemes is well known, that is, for an infinitesimal group scheme, there is a description, as algebras, of its coordinate ring (c.f. \cite[14.4]{W}). Similarly, in the graded case, the structure of positively graded, algebraically connected Hopf algebras of finite type (each degree finite dimensional) is also known (c.f. \cite[7.8]{MM}). In this section we describe the structure of a class of gr-group schemes which we define and call \emph{graded group varieties}. A graded group variety is a graded group scheme that somehow carries the characteristics of infinitesimal group schemes and algebraically connected Hopf algebra of finite type that we need to understand their structure. Its coordinate rings is gr-local, just like for infinitesimal group schemes the coordinate ring is local. Also, the coordinate ring is of finite type, like the graded Hopf algebras in \cite{MM}.

\begin{definition}
Let $G$ be a gr-group scheme, and let $A = \kf[G]$. If $A$ is gr-local, positively graded, of finite type we say that $G$ is a \emph{graded group variety (gr-group variety)}. Note that by \ref{gr-local} if $A$ is positively graded then, $A$ is gr-local is equivalent to $A_0$ local. 
\end{definition}

\begin{rk}[About the choice of name: gr-group variety]
In algebraic geometry, a variety is a scheme which is, in particular, connected and of finite type (in the geometric sense). Given a gr-group scheme $G$, we can associate the geometric object $\Specg(A)$, where $A = \kf[G]$. If $G$ is a gr-group variety, $A_0$ local implies that $\Specg(A)$ is connected (c.f. \ref{conn}). Therefore in this sense, gr-group varieties are connected and of finite type. 

Note that an infinitesimal gr-group scheme as defined in  \ref{defcon} is a gr-group variety, but the class of gr-group varieties is broader since it is not required for $\kf[G]$ to be finite dimensional, but only of finite type. 
\end{rk}

\subsection{Antipode for graded group varieties}

As part of our quest of classifying gr-group varieties, a natural question is: How much information do we need, in order to understand the graded Hopf algebra structure of a gr-group variety?

In our attempt to answer that question, we give a generalization to the well known result that states that, if $A$ is an algebraically connected, positively graded bialgebra then there exists an antipode $S$ constructed from the relations on $\e$ and $\Delta$, making $A$ into a graded Hopf algebra. See \cite[Prop. 8.3]{MM}.

\begin{theorem}\label{antipode}
Let $A$ be a positively graded, gr-local bialgebra of finite type, there exists an antipode map $S$ making $A$ into a graded Hopf algebra, that is, $A$ is the coordinate ring of a gr-group variety. 
\end{theorem}

\begin{proof}
Let $(A_0, \mathfrak{m})$ be a local ring. Note that $\e$ restricts to a surjective map on $A_0$, hence, $\mathfrak{m} = \ker(\e) \cap A_0$ and $A_0/\mathfrak{m} \cong \kf$. If $a \in A_0$ we can write $a = \lambda +m$ where $\lambda \in \kf$ and  $m \in \mathfrak{m}$.

We will describe $\Delta(a)$ for $a \in A$. First, let $a \in A_0$, then $\Delta(a) = \sum a_1 \otimes a_2$ where $|a_1| = |a_2| = 0$. The counit diagram gives us that $a = \sum \e(a_1)a_2 = \sum a_1 \e(a_2)$. 

Write $a_1 = \lambda_1+m_1$ and $a_2 = \lambda_2 +m_2$ where $\lambda_i \in \kf$ and $m_i \in \mathfrak{m}$, and note that $\e(\lambda_i + m_i) = \e(\lambda_i) = \lambda_i$. So we can write $a = \sum \lambda_1a_2 = \sum a_1 \lambda_2$. Therefore, for $a \in A_0$ 
$$\Delta(a) = a \otimes 1+ 1 \otimes a + \sum a_1 \otimes m_2 + \sum m_1 \otimes a_2.$$

If $m \in \mathfrak{m}$, we can write $\Delta(m) = m \otimes 1 + 1 \otimes m + \sum a_1 \otimes a_2$ and rewriting $a_i = \lambda_i + m_i$ by the counit map we get that $\sum a_1\lambda_2 = \sum \lambda_1 a_2 = 0$. Hence we can write $$\Delta(m) =  m \otimes 1 + 1 \otimes m  +  \sum m_1 \otimes m_2,$$ where $m_i \in \mathfrak{m}.$

If $a$ is nonzero degree then with similar arguments we can write $$\Delta(a) = a \otimes 1 + 1 \otimes a + \sum a_1 \otimes m_2 + \sum m_2 \otimes a_2 + \sum b_1 \otimes b_2,$$

where $m_i \in \mathfrak{m}$, $|a_i| = |a|$ and $|b_i| < |a|$. 

We can derive $S$ using that $(S \ast id)(a) = \e(a)1_A$, where $\ast$ is the convolution product for $\Homa(A,A)$.  

Without loss of generality, a basis for $A_0$ is of the form $\{1, m_1, \ldots, m_k\}$, then $S$ must satisfy

$$S(m_l) + m_l + \sum_{1 \leq i,j \leq k} \alpha_{lij}S(m_i)m_j = 0, $$

where $\alpha_{lij} \in \kf$. 

We can rewrite the equation as $$S(m_l)\left[1+n_l\right] + \sum_{\substack{i \neq l \\ j }} \alpha_{lij} S(m_i)m_j = -m_j,$$

where $n_l = \sum_j \alpha_{llj}m_j \in \mathfrak{m}$. 

Think of the above equation as a $(k \times k)$ system in $A$ where $S(m_1), \ldots, S(m_k)$ are the unknowns. The system corresponds to the matrix 
$$\begin{bmatrix}
1+n_1 & r_{12} & \cdots & r_{1n} \\
r_{21} & 1+n_{2} & \cdot & r_{2n}\\
\vdots & & \ddots & \vdots \\
r_{n1} & \cdots &  & 1+n_{k}
\end{bmatrix},$$

where $r_{li} = \sum_j \alpha_{lij}m_j$. By inspection, this matrix's determinant is in $1 + \mathfrak{m}$, hence, it is invertible and the system has a unique solution as desired. 

Let $\{a_1, \ldots, a_n\}$ be a basis for $A_k$ for $k > 0$,  then $S$ must satisfy
$$S(a_l) +a_l + \sum_{i = 1}^n \left[ S(a_i)(m_{li} +S(\widehat{m}_{li})a_i\right] + \sum S(b_1)b_2 = 0, $$
where $m_{li}, \widehat{m}_{li} \in \mathfrak{m}$ and $b_i$ are of lower degree. 

We can rewrite as $$S(a_l)(1+m_{ll}) + \sum_{i \neq l} S(a_i)m_{li} = - \left[ \sum_{i = 1}^n S(\widehat{m}_{li})a_i + \sum S(b_1)b_2 +a_l\right].$$

Note that all the terms in the right hand side are known by induction. Think of the above equation as an $(n \times n)$ system in $A$ where $S(a_1), \ldots, S(a_n)$ are the unknowns. The system corresponds to the matrix 
$$\begin{bmatrix}
1+m_{11} & m_{12} & \cdots & m_{1n} \\
m_{21} & 1+m_{22} & \cdot & m_{2n}\\
\vdots & & \ddots & \vdots \\
m_{n1} & \cdots &  & 1+m_{nn}
\end{bmatrix}.$$

Just as before, the system has a unique solution as desired. 
\end{proof}

As an interesting fact, theorem \ref{antipode} is an algebraic analogue to the following theorem in algebraic geometry, regarding abelian varieties. 

\begin{theorem}(From \cite[Appendix 4]{M})
Let $X$ be a complete variety, $e \in X$ a point, and 
$$m: X \times X \to X$$
a morphism such that $m(x,e) = m(e,x) = x$ for all $x \in X$. Then $X$ is an abelian variety with group laws and identity $e$. 
\end{theorem}

Rephrasing, this theorem says that if $X$ is a projective variety with multiplication and identity, then there exists an inverse making $X$ into a group scheme.

\subsection{Algebraic connectivization}

Most of the literature regarding graded Hopf algebras assume these algebras to be algebraically connected. This is not the case in this article. The next construction is a bridge between the classical setting and our broader class of graded Hopf algebras.

\begin{definition}\label{kappadef}
Let $A$ be a graded Hopf algebra. Let $\kappa(A) = A \otimes_{A_0} \kf$. We call $\kappa(A)$ the \emph{algebraic connectivization} of $A$. 
\end{definition}

\begin{theorem}
Let $A$ be a positively graded Hopf algebra. The algebraic connectivization of $A$, $\kappa(A)$ is an algebraically connected graded Hopf algebra.  
\end{theorem}

\begin{proof}
The counit and antipode maps for $\kappa(A)$ are given by $\kappa(\e) = \e \otimes_{A_0} \kf$ and $\kappa(S) = S \otimes_{A_0} \kf$. The comultiplication is given by the map $\kappa(\Delta)$ that makes the following diagram commute. 

$$\xymatrix{\kappa(A) \ar[r]^-{\Delta \otimes_{A_0} \kf} \ar[dr]_{\kappa(\Delta)} & A \otimes (A \otimes_{A_0} \kf) \ar[d] \\
& \kappa(A) \otimes \kappa(A)}$$

More explicitly if $a \in A$ we write $$\Delta(a) = \sum_{|a_1| = |a|}  a_1 \otimes a_2 + \sum_{|b_2| = |a|} b_1 \otimes b_2 + \sum_{|c_1|, |c_2| \neq |a|} c_1 \otimes c_2$$ where $|a_1|+|a_2| = |b_1|+|b_2| = |c_1|+|c_2| = |a|$ then 

$$\kappa(\Delta)(a) = \sum a_1 \otimes \e(a_2) + \sum \e(b_1) \otimes b_2 + \sum c_1 \otimes c_2.$$

With the counit, antipode and comultiplication as above, a quick chasing of diagrams shows that $\kappa(A)$ is a graded Hopf algebra over $\kf$. Note that its zeroth part is $\kappa(A)_0 = (A \otimes_{A_0} \kf)_0 = A_0 \otimes_{A_0} \kf  \cong \kf$. Hence $\kappa(A)$ algebraically connected and that $A \cong (A \otimes_{A_0} \kf) \otimes A_0  \cong \kappa(A) \otimes A_0$ as graded algebras. 
\end{proof}

\begin{lemma} Let $A$ be the coordinate ring of a gr-group variety. The connectivization $\kappa(A)$ is conormal, in the sense that there exists a short exact sequence of graded Hopf algebras
$$
\kf\rightarrow A_0\rightarrow A\rightarrow\kappa(A)\rightarrow\kf.
$$
\end{lemma}

\begin{proof} It can be easily checked that $A_0$ is the cotensor $A\hspace{-2ex}\qed_{\kappa(A)}\kf$.
\end{proof}

Since all evenly graded Hopf algebras are Hopf algebras, the previous lemma has the following remarkable corollary.

\begin{corollary} Let $A$ be a Hopf algebra which admits a non-trivial grading making it into the coordinate ring of a gr-group variety. The group $\Spec A$ has a non-trivial normal closed subgroup.
\end{corollary}

\begin{lemma}\label{freeind}
Let $C$ be a local ring $(C,\mathfrak{m})$ with algebra map $\Delta: C \to C \otimes C$ of the form $\Delta(a) = a \otimes 1 + 1 \otimes a + \sum a_1 \otimes a_2,$
where $a_1, a_2 \in \mathfrak{m}$. If $M$ is a finitely generated $C$-module with some map $\Delta_M: M \to M \otimes C$ such that for $a \in C$ and $x \in M$
\begin{itemize}
\item $\Delta_M(ax) = \Delta(a)\Delta_M(x)$, where $\Delta$ is the comultiplication on $C$, and
\item $\Delta_M(x) = x \otimes 1+x \otimes b + \sum b_1 x \otimes b_2$, where $b, b_1, b_2 \in \mathfrak{m}$
\end{itemize}
then $M$ is a free $C$-module. 
\end{lemma}

\begin{proof} Note that the first condition on the statement is saying that $M \otimes C$ is a $C$-module via $\Delta$ and that $\Delta_M : M \to M \otimes M \otimes C$ is a $C$-module morphism. 
The proof is by induction on the number of generators of $M$. 
 Let $\{x_1,\ldots,x_k\}$ be a minimal set of generators for $M$ over $C$. 
 
 Let $M  = Cx$. If $M$ is not free, then there exists $a \in \mathfrak{m}$ such that $ax=0$. For 
 $$\Delta_M(ax) = x \otimes a + x \otimes ab + \sum b_1 x \otimes a b_2 + \sum a_1 x \otimes a_2 + \sum a_1 x \otimes a_2b + \sum a_1 b_1 x \otimes a_2 b_2$$

to be zero either the term $x \otimes \_$ or $\_ \otimes a$ must be zero. The former is $a + ab$. If $a+ab = 0$, then $a = 0$, since  $(1+b)$ is invertible. The latter is 
$$x + \sum_{\textrm{if } a_2 = a} a_1x + \sum_{\textrm{if } a_2b = a}  a_1 x + \sum_{\textrm{if } a_2b_2 = a} a_1b_1x.$$

In any case, this is of the form $x(1+m)$ where $m \in \mathfrak{m}$, which is zero if and only if $x = 0$, this contradicts the fact that $x$ generates $M$ over $B$. 

Let $\{x_1, \ldots, x_k \}$ be a minimal set of generators for $M$ over $C$. If $M$ is not free over $C$, there must exist a relation of the form 
$$a_1 x_1 + \cdots + a_k x_k = 0. $$

By minimality, $a_i \in \mathfrak{m}$. Let $M' = M/(x_1)$; then $a_2 x_2 + \cdots + a_k x_k = 0$ in $M'$. Since $\Delta_M(x) = (x\otimes 1)(1 \otimes 1 + 1\otimes b + \sum b_1 \otimes b_2)$, the map $\Delta_M$ restricts to $M'$ and  has the same form as before, therefore, by induction, $a_2 = \cdots = a_k = 0$. Similarly, we can deduce $a_1 = 0$ by looking at $M/(x_2)$. Therefore $M$ is free over $C$ as desired.

\end{proof}

\begin{lemma}\label{Afree}
Let $A$ be the coordinate ring of a gr-group variety. Then $A$ is a graded free $A_0$-module.
\end{lemma}
\begin{proof}
Let $(A_0,\mathfrak{m})$ be local. Note that since $A$ is a positively graded Hopf algebra then each $A_n$ is an $A_0$-module. Also, since $A$ is a positively graded Hopf algebra $A_0$ is a sub-Hopf algebra of $A$.  

Consider $A_n$, as computed in the proof of \ref{antipode}, for $a \in A_0$ and $x \in A_n$ we have that,
$$
\Delta(a) = a \otimes 1 + 1 \otimes a + \sum a_1 \otimes a_2,
$$
where $a_1, a_2 \in \mathfrak{m}$. Similarly
$$
\Delta(x) = x \otimes 1+x \otimes b + \sum b_1 x \otimes b_2 + \sum c_1 \otimes c_2,
$$
where $|c_i| < n$ and $b, b_1, b_2 \in \mathfrak{m}$.

When restricting $\Delta$ to $\Delta_{A_n}: A_n \to A_n \otimes A_0$, $A_n$ is an $A_0$-module satisfying the conditions in \ref{freeind}. Therefore $A$ is a free $A_0$-module as desired. 
\end{proof}

\begin{theorem}\label{kappa}
Let $A$ be a coordinate ring of a gr-group variety, then $$A \cong \kappa(A) \otimes A_0$$ as graded algebras.
\end{theorem}

\begin{proof}
The theorem follows from Lemma \ref{Afree}.
\end{proof}

\subsection{Structure of graded group varieties}

We now classify the graded algebra structure of the coordinate rings of gr-group varieties over perfect fields of characteristic $p > 0$. For this we need to recall some known results from \cite{MM} and \cite{W}. 

\begin{proposition}(From \cite[7.8]{MM})

Let $A$ be a gr-commutative Hopf algebra over a field $\kf$ of characteristic  $p >0$, generated by one positively graded element $x$, then $A \cong \kf[x]^{gr}$, or $A \cong \frac{\kf[x]^{gr}}{(x^{p^l})}.$

\end{proposition}

\begin{proposition}(From \cite[14.4]{W}) \label{inf}
Let $G$ be an infinitesimal group scheme over a perfect field $\kf$ of characteristic $p > 0$, then its coordinate ring, $\kf[G]$, is isomorphic, as an algebra, to an algebra of the form:
$$\kf[G] \cong \frac{\kf[t_1, \ldots, t_n]}{(t_1^{p^{r_1}}, \ldots, t_n^{p^{r_n}})}.$$
\end{proposition}

\begin{proposition}(From \cite[7.11]{MM}) \label{grcon} 
Let $G$ be a positive, algebraically connected, gr-group scheme over a perfect field of characteristic $p$.  If its coordinate ring, $\kf[G]$, is of finite type, then is it is isomorphic, as an algebra, to an algebra of the form below. 

\begin{enumerate}

\item If $p > 2$ then, 
$$\kf[G] \cong
		\frac{\kf[x_1, \ldots, x_n, y_1, \ldots, y_m]^{gr}}{(x_1^{p^{l_1}}, \ldots, x_n^{p^{l_n}})}, $$
		where $|x_i|$ are even, $|y_i|$ are even or odd (of nonzero degrees). 
		
\item If $p = 2$ then,
$$\kf[G] \cong
	\frac{\kf[x_1, \ldots, x_n, y_1, \ldots, y_m]^{gr}}{(x_1^{2^{l_1}}, \ldots, x_n^{2^{l_n}})},$$
		where $|x_i|$ and $|y_i|$ are nonzero, and even or odd. 	
\end{enumerate}
\end{proposition}

\begin{theorem}\label{grinf}
Let $G$ be gr-group variety over a perfect field of characteristic $p$. Then its coordinate ring, $\kf[G]$, is isomorphic as an algebra to an algebra of the form below. 

\begin{enumerate}

\item If $p > 2$, then 
$$\kf[G] \cong
		 \frac{\kf[x_1, \ldots, x_n, y_1, \ldots, y_m]^{gr}}{(x_1^{p^{l_1}}, \ldots, x_n^{p^{l_n}})}, $$
		where $|x_i|$ are even, $|y_i|$ are even or odd (including zero degrees). 
		
\item If $p = 2$, then
$$\kf[G] \cong
		\frac{\kf[x_1, \ldots, x_n, y_1, \ldots, y_m]^{gr}}{(x_1^{2^{l_1}}, \ldots, x_n^{2^{l_n}})},$$
		where $|x_i|$ is odd or even (including zero) and $|y_i|$ is even or odd and  nonzero.  	
		
\end{enumerate}
\end{theorem}

\begin{proof}
Given $G$ as above then $A = \kf[G]$ contains no nontrivial idempotent hence the degree zero part $A_0$ is a connected Hopf algebra in the ungraded sense, hence $A_0$ represents an infinitesimal group scheme. Then $A_0$ is as in \ref{inf}.  Consider, $\kappa(A) = A \otimes_{A_0} \kf$ as defined in \ref{kappadef}; then $\kappa(A)$ is as in \ref{grcon}.  By \ref{kappa}, $A \cong  \kappa(A) \otimes_{\kf} A_0$ as graded algebras, hence the result follows.  
\end{proof}

\section{Graded connected and graded separable components}\label{secgrcon}

Given a positive, gr-group scheme $G$, we define the gr-connected component $G^0$, and  the gr-group of connected components $\pi_0 G$ associated to it. 

\begin{definition}
Let $A = \kf[G]$, we define $\pi_0A$ to be the largest gr-separable subalgebra of $A$. By \ref{piA}, if $A$ is a positively graded $\kf$ algebra ($\kf$ not graded), then $\pi_0 A = \pi_0 A_0$ where $A_0$ is the degree zero part of $A$. Hence 
$$\pi_0G(R) = \Homg(\pi_0 A, R) = \Homa(\pi_0 A_0, R) = \Homa(\pi_0 A_0, R_0).$$ In fact, by gr-separable elements over an ungraded field are necessarily trivially graded.

The gr-group scheme corresponding to $\pi_0 A$ is denoted by $\pi_0G$ and called the \textit{gr-group of connected components}. (Graded separable extensions are defined in \ref{gr-separable}.)
\end{definition}

\begin{definition}
The inclusion $\pi_0A \subset A$ gives a map $G \to \pi_0G$. Let $G^0$ be the gr-group scheme corresponding to the kernel of this map, $G^0$ is the \textit{gr-connected component} of $G$. 
\end{definition}

\begin{theorem}\label{pides}
Let $G$ be a positive, gr-group scheme with coordinate ring $A = \kf[G]$. If $A_0$ is finite dimensional then, $\pi_0 A = \prod_i \kf_i^s$, where $\kf_i^s$ denotes the separable closure of $\kf$ in $\kf_i$.
\end{theorem}

\begin{proof}
Note that since $A_0$ is finite dimensional $A_0 = \prod_i A_0^i$ where each $A_0^i$ is a local ring. Then $\pi_0 A = \pi_0A_0$. Then $\pi_0A_0 = \prod_i \pi_0 A_0^i$. 

To see that, let $s \in A_0^i$ separable over $\kf$; then $(0, \ldots, 0, s, 0, \ldots)$ is in $A_0$ and separable over $\kf$; therefore $\prod_i \pi_0 A_0^i \subseteq \pi_0 A_0$. 

Let $s = (s_i) \in \pi_0 A_0$; then $s(0, \ldots, 0, 1, 0, \ldots, 0)  = s_i \in \pi_0A_0$ since $s$ and \phantom{-} $(0, \ldots, 0, 1, 0, \ldots, 0)$ are in $\pi_0 A_0$. Hence the claim is true. 
 
So now it is enough to show that if $(A_0, \mathfrak{m})$ is a finite local ring with residue field $\widetilde{\kf}$ then $\pi_0 A_0 = \widetilde{\kf}^s$ where $\widetilde{\kf}^s$ denotes the separable closure of $\kf$ in $\widetilde{\kf}$. By definition $\pi_0 A_0$ is a separable subalgebra of $A$. Hence by the classification of separable algebras, $\pi_0 A_0$ is a product of matrix rings over division rings whose centers are finite dimensional separable field extensions of $\kf$.  Since $A$ is commutative it follows that $\pi_0A_0$ is a product of separable field extensions of $\kf$.  Since $A_0$ is local and $\pi_0 A_0 \subset A$, then it follows that $\pi_0 A_0$ is exactly a separable field extension of $\kf$. Every element in $\pi_0 A_0 - \{0\}$ is invertible hence it survives under the quotient $A_0/\mathfrak{m} \cong \widetilde{\kf}$; hence $\pi_0 A_0$ is a separable field extension of $\kf$ contained in $\widetilde{\kf}$.  Thus $\pi_0 A_0 \subseteq \widetilde{\kf}^s$.

Now by \cite[6.8]{W} we have that if $(A_0, \mathfrak{m}, \widetilde{\kf})$ is finite dimensional local, then $\pi_0 A  = \pi_0A_0 \cong \pi_0(A_0/\mathfrak{m}) \cong \pi_0(\widetilde{\kf}) = \widetilde{\kf}^s$, hence $\pi_0 A= \widetilde{\kf}^s$. 
\end{proof}

When $G$ be a positive, finite, gr-group scheme, we have a more explicit description for $G_0$. Let $A = \kf[G]$, by \ref{hens} we write $A = \prod_{i=1}^n A^i$ where $(A^i, \mathfrak{m}_i)$ are gr-local rings
with residue gr-field $\kf_i$. Let $\kf \subset \kf_i^s \subset \kf_i$, where $\kf_i^s$ is the gr-subfield of $\kf_i$
 whose homogeneous elements are separable over $\kf$.  

For a finite gr-group scheme $G$ and $A=\prod_i A^i$ as above, let $e_i$ be the identity of $A^i$. 
The counit map $\e: A \to \kf$ sends all $e_i$ but one to $0$, say $e_0$. To see this, notice that $\e(1_A) = 1$ and $1_A = \sum_i e_i$. 
Let $\e(e_i) = \lambda_i \in \kf$; since the $e_i$ are idempotents it follow that 
$\e(e_i) = \e(e_i^2)$ thus $\lambda_i = \lambda_i^2$, which in the case of a field implies that $\lambda_i$ is either zero or one. 
Now since $\e(1_A) = \e( \sum_i e_i) = \sum_i \lambda_i = 1$  then exactly one of the $\lambda_i$, say $\lambda_0$ must be nonzero and hence equal to one. Then $\e$ factors through $A^0$, $\e:A \to A^0 \to \kf$.

Note that since $\e:A \to \kf$ factors through $A^0$ and $\e$ is surjective then $\e: A^0 \to \kf$ is surjective. We know that $\kf \subset \kf_0$ and since $\widehat{\e}: A^0/\mathfrak{m}_0  = \kf_0 \to \kf$ is surjective
it follows by Schur's Lemma that $\kf_0  = \kf$. Therefore  $A^0$ is a gr-local algebra with residue gr-field equal to $\kf$. 

\begin{theorem}\label{grcomp}
Let $G$ be a positive, finite, gr-group scheme. Then the coordinate ring for $G^0$ is $A^0$ as above. 
\end{theorem}

\begin{proof}
By \ref{pides} $\pi_0 A = \prod_i \kf_i^s$, the map $G \to \pi_0 G$ is given by the inclusion $\pi_0 A = \prod_i {\kf}_i^s \subset A$. By \cite[2.1]{W} the kernel of $G \to \pi_0 G$ is represented by  $A \otimes_{\pi_0 A} \kf = A /(I \cap \pi_0 A) A$ where $I = \ker \e$
is the augmentation ideal. We have that $\prod_{i \neq 0} {\kf}_i^s \subset I$. We then get that 
$I \cap \pi_0 A = \prod_{i \neq 0} {\kf}_i^s$. 
Therefore $A/(I \cap \pi_0 A)A = A/(\prod_{i \neq 0} {\kf}_i^s A) = (\prod_i A^i)/(\prod_{i \neq 0} {\kf}_i^s A^i) = \prod_i A^i/ \prod_{i \neq 0} A^i  \cong A^0$, which is the algebra that corresponds to $G^0$. Thus, $A^0$ is the coordinate ring for $G^0$.
\end{proof}

\subsection{The graded spectrum and connectivity}
The graded (prime) spectrum, denoted $\Specg(R)$, for a graded ring $R$, is defined and studied in Appendix \ref{local}. We give results regarding the graded spectrum of the coordinate ring of a positive, gr-group scheme.

\begin{proposition}\label{conn}
Let $G$ be a positive, gr-group scheme with $A = \kf[G]$ and $A_0$ finitely generated. Then $\pi_0 G$ is trivial if and only if $\Specg(A)$  is connected. 
\end{proposition}

\begin{proof}
Since $G$ is positive, $\pi_0 A = \pi_0 A_0$. Then by \cite[6.6]{W} $\pi_0 G$ is trivial if and only if $\Spec(A_0)$ is connected and by \ref{grconcon}, $\Spec(A_0)$ is connected if and only if $\Specg(A)$ is connected.

\end{proof}

\begin{definition}\label{defcon}
A gr-group scheme $G$ is \textit{connected} if $\pi_0 G$ is trivial. By \ref{idem} $G$ is connected if and only if $\kf[G]$ contains no nontrivial idempotents. 
\end{definition}

\begin{rk}
If $G$ is algebraically connected, that is, the zero degree part of $\kf[G]$ is $\kf$ then $G$ is connected since $\kf[G]$ contains no nontrivial idempotents. 
\end{rk}

\begin{definition}
A finite gr-group $G$ is \textit{\'{e}tale} if $\kf[G]$ is separable. In this case, by \ref{piA} it follows that if $G$ is positive and \textit{\'{e}tale} then $\kf[G]$ must be trivially graded. 
\end{definition}

\section{Classification of finite graded group schemes}\label{secclas}

For the next result we can follow the proof of \cite[6.8]{W} to get the graded version. 

\begin{proposition}[From \cite{W}]
Let $G$ be a finite, positive, gr-group scheme over a perfect field. Then $G$ is the semi-direct product of $G^0$ and $\pi_0G$. 
\end{proposition}
\begin{proof}
Let $A = \kf[G]$. Since $A$ is a product of gr-local rings  $A = \prod_i A^i$,  the nilradical of $A$ is $N = \prod_i \mathfrak{m}_i$ and each $\mathfrak{m}_i  = (A^i)^+ \oplus (\mathfrak{m}_i)_0$, that is, each $\mathfrak{m}_i$ is the irrelevant ideal of $A^i$ plus the zeroth part of $\mathfrak{m}_i$. Then $A/N = A/(N_0 \oplus A^+) = A_0/N_0$, and since $A/N$ is separable (since $\kf$ is perfect), then by \cite[Cor. 6.8]{W} $\pi_0 (A) = \pi_0 (A_0) \cong A_0/N_0 = A/N$. Hence, as in \cite[6.8]{W}, $A/N$ defines a graded subgroup scheme of $G$ which is isomorphic to $\pi_0 G$.  

We want the last map in the exact sequence 
$$\xymatrix{0 \ar[r] & G^0 \ar[r] & G \ar[r] & \pi_0 G}$$ to be surjective. By \cite[6.8]{W} if we look at the zero part of $A$, $A_0$ we have that for a graded algebra $R$, 

$$\xymatrix{\Homa(A_0, R_0) \ar@{->>}[r] & \Homa(\pi_0 A_0, R_0) = \Homg(\pi_0 A, R)}$$ is surjective. Let $ f \in \pi_0 G(R)$, then by the surjectivity there is a map corresponding to $f$, say $g \in  \Homa(A_0, R_0) $. Then we can define $\widehat{g} \in G(R)$ so that it is $g$ in the zero component and zero elsewhere, hence $G(R)  = \Homg(A, R) \twoheadrightarrow \pi_0 G(R) = \Homg(\pi_0A, R)$ is surjective. 

Notice that the map $$\xymatrix{\pi_0 G \ar@{^(->}[r] & G \ar@{->>}[r] & \pi_0 G}$$ corresponds to the composition of maps $$\xymatrix{\prod\kf^s_i \ar[r] & A \ar[r] & A/N\cong \prod\kf^s_i}$$ which is the identity map.
\end{proof}

\begin{theorem}
Let $G$ be an abelian, finite, positive, gr-group scheme over a perfect field, then $G$ splits canonically into four factors of the following types:

\begin{enumerate}
\item \'{e}tale with \'{e}tale dual, 
\item \'{e}tale with connected dual, 
\item connected with \'{e}tale dual,
\item connected with connected dual.
\end{enumerate}
\end{theorem}

\begin{proof}
Since $G$ is abelian we have that $G = G^0 \times \pi_0 G$ and when taking duals we have the following decomposition, $$G \cong G^{\sharp \sharp} \cong ((G^\sharp)^0 )^\sharp\times (\pi_0 (G^\sharp))^\sharp.$$ 
Applying this decomposition to both $G^0$ and $\pi_0 G$ we get that $$G^0 \cong ((G^0)^\sharp)^0)^\sharp \times (\pi_0((G^0)^\sharp))^\sharp$$
 which is a product of connected with connected dual and connected with \'{e}tale dual. Similarly $$\pi_0 G \cong ((\pi_0(G^\sharp))^0)^\sharp \times (\pi_0(\pi_0(G^\sharp))^\sharp$$ which is a product of \'{e}tale with connected dual and \'{e}tale with \'{e}tale dual. 
\end{proof}

\subsection{Examples of decomposition}

We provide some examples of the decomposition of $G = G^0 \times \pi_0 G$ for finite abelian (gr)-group schemes.

\begin{example}[Ungraded case; characteristic zero]
Consider $A = \R[x]/(x^3-1)$ where $|x| = 0$, $\e(x) = 1$, $\Delta(x) = x \otimes x$ and $S(x) = x^2$. By the Chinese Remainder Theorem we can write $A$ as a product of local rings, 
 $$A \cong \R[x]/(x-1) \times \R[x]/(x^2+x+1) \cong \R \times \C.$$ We follow the isomorphism $$A \cong \R[x]/(x^3-1) \cong \R[x]/(x-1) \times \R[x]/(x^2+x+1)$$
  by finding nontrivial idempotents of $A$. We find that two nontrivial idempotents are $e = 1/3(x^2+x+1)$ and $(1-e) = -1/3(x^2+x -2)$; note that $\e(e) = 1$ and $\e(1-e) = 0$. Then $Ae \cong \R[x]/(x-1)$,  $A(1-e) \cong  \R[x]/(x^2+x+1)$ and, $A^0 \cong \R[x]/(x-1) \cong \R$ where $A^0$ is as in \ref{grcomp}. Hence $A^0$ is a sub-Hopf algebra with group scheme $G^0(R) = \{e\}$ the trivial group . Now $A \otimes \C = \C \times \C \times \C$, hence $A$ is separable and $\pi_0 A = A$, then $\pi_0 G(R) =  \mu_3(R) = \{ r \in R \,\, | \,\, r^3 =1\}$ and  $G(R) = \mu_3(R) = (G^0 \times \pi_0 G)(R)$. 
\end{example}

\begin{example}[Ungraded case; finite characteristic]
Consider $B= \F_2[x]/(x^3-1)$ where $|x| = 0$,  $\e(x) = 1$, $\Delta(x) = x \otimes x$ and $S(x) = x^2$. By the Chinese Remainder Theorem
 $$B \cong \F_2[x]/(x-1) \times \F_2[x]/(x^2+x+1) \cong \F_2 \times \F_2(\zeta)$$
  where $\zeta$ is a primitive $\nth{3}$ root of unity. We follow the isomorphism 
  $$B \cong \F_2[x]/(x^3-1) \cong \F_2[x]/(x-1) \times \F_2[x]/(x^2+x+1)$$ by finding nontrivial idempotents of $B$. We find that two nontrivial idempotents are $e = (x^2+x+1)$ and $(1-e) = (x^2+x)$; note that $\e(e) = 1$ and $\e(1-e) = 0$. Then $Be \cong \F_2[x]/(x-1)$,  $B(1-e) \cong  \F_2[x]/(x^2+x+1)$ and, $B^0 \cong \F_2[x]/(x-1) \cong \F_2$. Hence $B^0$ is a sub-Hopf algebra with group scheme $G^0(R) = \{e\}$ the trivial group. Now $B \otimes \overline{\F}_2 = \overline{\F}_2 \times \overline{\F}_2 \times \overline{\F}_2$, hence $B$ is separable and $\pi_0 B = B$, then $\pi_0 G(R) =  \mu_3(R) = \{ r \in R \,\, | \,\, r^3 =1\}$ and  $G(R) = \mu_3(R) = (G^0 \times \pi_0 G)(R)$. 
\end{example}

\begin{example}[Graded case; finite characteristic]
Let $C = \F_2[x,y]^{gr}/(x^3-1, y^2)$ where $|x| = 0$, $|y| = 1$,  $\e(x) = 1, \e(y) = 0$, $\Delta(x) = x \otimes x$ and $\Delta(y) = y \otimes 1 + 1 \otimes y$. 
Since tensor commutes with products we have that 
{\setlength\arraycolsep{1pt}
\begin{eqnarray*}
C   & \cong  &\left( \F_2[x]^{gr}/(x-1) \times \F_2[x]^{gr}/(x^2+x+1) \right) \otimes \F_2[y]^{gr}/(y^2)\\
 &   \cong &\F_2[x,y]^{gr}/(x-1, y^2) \times \F_2[x, y]^{gr}/(x^2+x+1, y^2). \end{eqnarray*}}
 This isomorphism is given by two nontrivial idempotents of $C$. By \ref{idem} the idempotents of $C$ are homogeneous of degree zero, therefore the only trivial idempotents of $C$ are, like in the previous example,
$e = (x^2+x+1)$ and $(1-e) = (x^2+x)$, $Ce \cong \F_2[x,y]^{gr}/(x-1,y^2)$ and $C(1-e) \cong \F_2[x,y]^{gr}/(x^2+x+1, y^2)$. Now $\pi_0 C \cong \F_2[x]^{gr}/(x^3-1)$ and $C^0 = \F_2[x,y]^{gr}/(x-1, y^2)$, also $\pi_0 G(R) = \mu_3(R_0)$ where $R_0$ is the degree zero part of a commutative graded ring $R$. Now $\F_2[x]^{gr}/(x-1) \cong \F_2$ and $\F_2[x]^{gr}/(x^2+x+1) \cong \F_2(\zeta)$, where $\zeta$ is a $\nth{3}$ primitive root, are  local rings. Hence $C \cong (\F_2 \times \F_2(\zeta)) \otimes \F_2[y]^{gr}/(y^2) \cong \F_2[y]^{gr}/(y^2) \times \F_2(\zeta)[y]^{gr}/(y^2)$ where each term is gr-local, both with unique homogeneous maximal ideal $(y)$. Now $G^0(R) = \nil^2_1(R) = \{ r \in R_1 \,\, |\,\, r^2 = 0\}$ with group structure given by sum.  Now $G(R) = (G^0 \times \pi_0 G)(R)$. Let $f, g \in G(R)$, then $f$ and $g$ are determined by what they sent $x$ and $y$ to, say $f(x) = r_1, g(x) = r_2 \in \mu_3(R_0)$ and $f(y) = s_1, g(y) = s_2 \in \nil^2_1(R)$. Then $(f \ast g)(x) = r_1r_2$ and $(f \ast g)(y) = s_1 +s_2$, hence $G(R) = \nil^2_1(R) \times \mu_3(R_0)$ where the multiplication is given by $(s_1, r_1) \cdot (s_2, r_2) = ( s_1 + s_2, r_1r_2)$ and $(s,r)^{-1} = (s, r^2)$ which corresponds to the antipode map $S: C \to C$ where $S(x) = x^2$ and $S(y) = -y$.  
\end{example}

\begin{example}[Dual Steenrod subalgebra $A(1)$]
For $A(1)= \F_2[\xi_1, \xi_2]^{gr}/(\xi_1^4, \xi_2^2)$, $\pi_0 A(1) = \F_2$ hence $\pi_0 G = \{e\}$. In fact $A(1)$ is a gr-local algebra with graded maximal ideal $(\xi_1, \xi_2)$ and $A^0 = A(1)$ therefore $G \cong G^0$. 
\end{example}

\appendix

\section{Graded spectrum and graded local rings}\label{local}

\begin{definition}\label{grlo}
 A graded ring $R$ is said to be \textit{gr-local} if there exists a unique homogeneous maximal ideal $\mathfrak{m}$. 
\end{definition}

\begin{definition}Let $R$ be a positively graded ring, then $R_+ = \sum_{i >0} R_i$ is a homogeneous ideal and it is called the \textit{irrelevant ideal}.
\end{definition}

From the definition above, we get the following observation. 
\begin{rk}
Let $R$ be a positively graded ring, for all ideals $I$ of $R_0$, $I \oplus R_+$ is a homogeneous ideal for $R$. 
\end{rk}

The following proposition is a consequence of the previous remark.  
\begin{proposition}\label{gr-local}
Let $R$ be a positively graded ring, then $R_0$ is a local ring if and only if $R$ is gr-local.
\end{proposition}

\begin{proposition}
If $R$ is positively graded and local (in the ungraded sense), then $R$ is gr-local.
\end{proposition}
\begin{proof}
Let $\mathfrak{m}$ be the unique maximal ideal of $R$. Therefore $R_+ \subset \mathfrak{m}$. Then $\mathfrak{m}$ is necessary homogeneous
since $\mathfrak{m} = (R_0 \cap \mathfrak{m}) \oplus R_+$. 
\end{proof}

\begin{rk}
A graded ring may be gr-local but not graded and local. For example, consider $R = \kf[x]$ with $|x| =1$, then 
$R_+$ is the unique maximal homogeneous ideal of $R$ but $R$ is not local. 
The reason is that any other maximal ideal of $R$ is not homogeneous.  
\end{rk}

\begin{definition}
 Let $R$ be a graded ring and $M$ a graded $R$-module. Then the \textit{graded Jacobson radical} of $M$ denoted 
by $J^g(M)$ is the intersection of all gr-maximal submodules of $M$.
\end{definition}

\begin{rk}
By \cite[2.9.1.vi]{NV} for a graded ring $R$, $J^g(R)$ is the largest proper ideal $I$ such that any $a \in R$ homogeneous is 
invertible, if the class of $a$ in $R/I$ is invertible. In the case of $R$ a gr-local ring with gr-maximal ideal 
$\mathfrak{m}$ then $J^g(R) = \mathfrak{m}$. Hence for $R$ gr-local, $R/\mathfrak{m}$ is a gr-division ring.
When $R$ is commutative we get that $R/\mathfrak{m}$ is a gr-field. We called that field \textit{the residue gr-field} of $R$.
\end{rk}

\begin{definition}(From \cite[2.11]{NV})
Let $R$ be a graded ring a graded ideal $P$ of $R$ is \textit{gr-prime} if $P \neq R$ and for graded ideals $I$ and $J$ of $R$ we have $I \subset P$ or $J \subset P$ only when $IJ \subset P$. 
\end{definition}

\begin{definition}(From \cite[2.11]{NV})
 The set of all gr-prime ideals of $R$ is denoted by $\Specg(R)$ and it is called the graded (prime) spectrum of $R$.
 \end{definition}
 
 \begin{proposition}(From \cite[II.2.11]{CO})\label{idem}
The idempotents of a graded ring $R$ are homogeneous of degree zero. 
\end{proposition}
 
 \begin{proposition}\label{grconcon}
 Let $R$ be a graded ring, then $\Specg(R)$ is connected if and only if $\Spec(R_0$) is connected, where $R_0$ is the degree zero part of $R$. 
 \end{proposition}
 \begin{proof}
 Note that $\Specg(R)$ is connected if and only if $R$ contains no nontrivial idempotents, then by \ref{grcon} it is equivalent to $R_0$ having  no nontrivial idempotents which is the case if and only if $\Spec(R_0)$ is connected.
 \end{proof}

\begin{proposition}\label{locon}
Let $R$ be a local ring, respectively gr-local, then $\Spec(R)$, respectively $\Specg(R)$, is connected. 
\end{proposition}
\begin{proof}
Let $P$ and $Q$ be gr-prime ideals such that $P+Q = R$ and $P \cap Q = (0)$ then there exists a (gr-)maximal ideals $\mathfrak{m}_1$ and $\mathfrak{m}_2$  such that $P \subset \mathfrak{m}_1$ and 
$Q \subset \mathfrak{m}_2$. Now if $R$ is (gr-)local then $\mathfrak{m}_1 = \mathfrak{m}_2 = \mathfrak{m}$ is the unique (gr-)maximal ideal, hence $R = P + Q \subset \mathfrak{m}$ which is a contradiction. 
\end{proof}

\section{Graded Henselian rings}

\begin{proposition}\label{dim}
 Any finitely generated gr-module $M$ over a gr-division ring $R$ has a well defined notion of dimension and $M \cong \sum_{i =1}^n R(a_i)$ where
$a_1, a_2, \ldots, a_n$ is a minimal set of homogeneous generators or a basis for $M$. We then denote the dimension by 
$\dim_R(M)$. 
\end{proposition}
\begin{proof}
We claim that a homogeneous set of elements is linearly independent 
if and only if they are `homogeneously' linearly independent. That is, 
$\sum_{i = 1}^n r_i a_i = 0$ for $r_i \in R$ implies that $r_i = 0$ for all $i$ if and only if $\sum_{i=1}^n s_ia_i = 0$ for $s_i$ homogeneous in $R$ implies that $s_i = 0$. 

One direction is clear. Now assume that $a_1, \ldots, a_n$ are homogeneously independent, then let $\sum_{i=1}^n r_i a_i = 0$ where $r_i$ are not 
necessarily homogeneous, we can rewrite this sum in terms of homogeneous elements by $\sum_{i= 1}^n r_i a_i = \sum_{i=1}^n \sum_{j \in \Z} (s_i)_{j-|a_i|}a_i = 0$. 
Each $\sum_{j \in \Z}(s_i)_{j-|a_i|}a_i$ is of a different degree, hence for the whole sum over $i$ to be zero we need each 
$\sum_{j \in \Z}(s_i)_{j-|a_i|}a_i = 0$. Each $\sum_{j \in \Z}(s_i)_{j-|a_i|}a_i$ is a homogeneous linear combination of the $a_i$'s, hence
each $(s_i)_{j-|a_i|} = 0$, which gives us that each $r_i =0$ as desired.
\end{proof}

\begin{rk}
 The above proposition allows us to proceed as usual and work with homogeneous
linear combinations where the `scalars' are homogeneous therefore when nonzero they are invertible. 
\end{rk}

\begin{proposition}\label{artin}
 Any finite graded algebra over a graded field is gr-artinian.
\end{proposition}
\begin{proof}
Let $A$ be a finite graded algebra over the gr-field $R$. Then by \ref{dim} we have $\dim_R (A)$ is finite.
 Let $\cdots \subset I_2 \subset I_1 \subset A$ be a descending chain of ideals. Then by \ref{dim} we have a chain of inequalities
$$ \cdots \leq \dim_R I_2 \leq \dim_R I_1 \leq \dim_R A.$$ This way we get a decreasing chain of inequalities that cannot continue indeterminately hence 
the descending sequence of graded ideals stabilizes. 
\end{proof}

\begin{definition}
 A commutative gr-local ring $R$ is \textit{gr-henselian} if every commutative finite graded $R$-algebra is \textit{gr-decomposed}, 
that is, if it is the direct sum of gr-local rings.  
\end{definition}

\begin{rk}\label{gr-hen}
 By \cite[II.3.14]{CO} a gr-field is gr-hensenlian. This follows from \ref{artin}. 
\end{rk}

\begin{corollary}[From Remark \ref{gr-hen}]
Any commutative finite graded algebra over a gr-field is the direct sum of gr-local rings.
\end{corollary}

\begin{corollary} \label{hens} If $A$ is a commutative finite graded algebra over a field $\kf$, then $A = \prod_{i=1}^n A_i$, where each $A_i$ is a gr-local ring 
where $\kf_i$ are the residue gr-field for $A_i$ with gr-maximal ideal $\mathfrak{m}_i$. 
Moreover the homogeneous elements of the gr-maximal ideals $\mathfrak{m}_i$ are nilpotent. \cite[A.3]{W}
\end{corollary}

\section{Graded separable extensions}

\begin{definition}\label{gr-separable}
 Let $L$ be a finite graded field extension of the graded field $K$. Then $L$ is a \textit{graded separable extension}
of $K$ if, for each homogeneous element $l \in L$, the minimal homogeneous polynomial of $l$ over $K$ has distinct roots.
\end{definition}

\begin{definition}
A finite dimensional graded algebra $A$ is a \textit{gr-separable} algebra if it is the product of gr-separable field extensions. 
\end{definition}

\begin{rk}
Let $\kf$ be a field. A graded field extension of $\kf$ may be one of the following; $L = l$ where $l$ is a field extension in the usual sense, or $L = l[X, X^{-1}]$ where $l$ is a field extension in the usual sense. 
\end{rk}

\begin{rk}\label{piA}
Any gr-field of the form $L[T, T^{-1}]$ cannot be a gr-separable extension for $\kf$ since $T$ is transcendental over $\kf$, even if $L$ is a separable extension of $\kf$. 
\end{rk}


\bibliography{../resources/biblio}
\bibliographystyle{amsalpha}

\end{document}